\documentclass[12pt,oneside,mathscr]{amsart}

\usepackage{amsthm,amsmath,amssymb,euscript,amscd,a4wide}

\usepackage[all]{xy}
\usepackage{multirow,bigstrut}
\usepackage{tikz}
\newtheorem{thm}{Theorem}[section]
\newtheorem{cor}[thm]{Corollary}
\newtheorem{lemma}[thm]{Lemma}
\newtheorem{prop}[thm]{Proposition}
\newtheorem{quotethm}[]{Theorem}

\theoremstyle{definition}
\newtheorem{dfn}[thm]{Definition}
\newtheorem{note}[thm]{Notes}
\newtheorem{example}[thm]{Example}
\newtheorem{conjecture}[thm]{Conjecture}

\newtheorem{remark}[thm]{Remark}
\newtheorem{notation}[thm]{Notation}
\numberwithin{equation}{section}

\newcommand{\C}{\mathbb{C}}
\newcommand{\Z}{\mathbb{Z}}
\newcommand{\R}{\mathbb{R}}
\newcommand{\Q}{\mathbb{Q}}
\newcommand{\OO}{{\mathcal O}}
\newcommand{\NN}{\operatorname{N}}
\newcommand{\NS}{\operatorname{NS}}
\newcommand{\td}{\operatorname{Td}}

\newcommand{\A}{\mathcal{A}}
\newcommand{\stab}{\operatorname{Stab}}
\newcommand{\free}{\operatorname{free}}
\newcommand{\coh}{\operatorname{Coh}}
\newcommand{\Pic}{\operatorname{Pic}} 
\newcommand{\Coh}{\operatorname{Coh}}
\newcommand{\ch}{\operatorname{ch}}
\newcommand{\T}{{\mathbb T}}
\renewcommand{\P}{{\mathcal P}}
\newcommand{\xhat}{{\Hat x}}
\newcommand{\I}{{\mathcal I}}
\newcommand{\Plane}[1]{\Pi_{#1}}
\newcommand{\ray}[1]{\Xi_{#1}}

\newcommand{\spc}[1]{\Omega_{#1}}
\newcommand{\compo}{\raise2pt\hbox{$\scriptscriptstyle\circ$}}
\renewcommand{\leq}{\leqslant}
\renewcommand{\geq}{\geqslant}

\renewcommand{\Re}{\operatorname{Re}}
\renewcommand{\Im}{\operatorname{Im}}

\begin{document}

\title[Walls for Bridgeland Stability]{Computing the Walls Associated
  to Bridgeland Stability Conditions on Projective Surfaces}
\author[A Maciocia]{Antony Maciocia}
\address{Department of Mathematics and Statistics\\
The University of Edinburgh\\
The King's Buildings\\ Mayfield Road\\ Edinburgh, EH9 3JZ.\\}
\email{A.Maciocia@.ed.ac.uk}
\thanks{} 
\date{\today}
%\subjclass{14F05, 14D20, 14J60, 18E30, 14N35, 14C20}
%\keywords{Moduli space, Bridgeland stability, walls}
\begin{abstract}
We derive constraints on the existence of walls for Bridgeland
stability conditions for general projective surfaces. We show that in
suitable planes of stability conditions the walls are bounded and
derive conditions for when the number of walls is globally finite.
In examples, we show how to use the explicit conditions to locate
walls and sometimes to show that there are no walls at all.
\end{abstract}

\maketitle

\section*{Introduction}{}
Throughout this paper we let $X$ be a smooth complex projective surface.
The notion of stability condition for a triangulated category (and, in
particular, for the derived category of coherent sheaves on $X$) was
introduced by Bridgeland (see \cite{Bri07}). Bridgeland went on in
\cite{BrK3} to construct explicit families of stability conditions on
K3 and abelian surfaces. Arcara and  Bertram (\cite{Arcara07}) then showed
that these examples work for all smooth complex projective surfaces. We
shall be exclusively interested in these examples and their associated
walls. We will call these \textbf{basic} (Bridgeland) stability
conditions.

There are various reasons for studying such stability conditions now
well documented in the literature. For example, the preservation of
stability under Fourier-Mukai transforms correspond to the non-existence 
of walls. The G\"ottsche conjecture and Donaldson-Thomas invariants
are (conjecturally) related to wall crossing formula in such stability spaces.

One particular  use is to help us to compute
explicit moduli spaces of objects in the derived category of a
variety. At a boundary of the space of stability conditions the
stability condition becomes the more familiar one of (twisted)
Gieseker stability. We can construct these spaces using techniques of Geometric
Invariant Theory and often they are related to other moduli spaces
such as moduli spaces of connections satisfying certain differential
equations. But for fixed characteristic classes, the space of
stability conditions is divided into chambers and once we
cross a wall into a different chamber the objects representing points
of the corresponding moduli space need not be torsion-free sheaves or
indeed be sheaves at all. In order to better understand these moduli
spaces (including their existence) our first step is to locate the
walls separating the chambers. These walls arise from a combination of
numerical constraints and constraints coming from the hom sets of
objects in the derived category. Indeed, understanding better which
walls can exist should help us understand the derived category and
its sets of maps. The numerical constraints turn out to be
essentially quadratic in suitable coordinates in the stability
manifold $\stab(X)$ and so are (real) quadrics. If we restrict to
certain real half-planes in $\stab(X)$ given by pairs
$(\A_{\beta,\omega},Z_{\beta,\omega})$ (see section \ref{s:dist} for a
detailed description of this)
then these are conic sections
with centres on the boundary of the half-planes. We will prove that,
provided we choose the plane correctly, these are heavily constrained
to be \textbf{nested}. By this we mean that the conics lie entirely
inside each other so that they can never intersect or lie in two halves
of the plane divided by a straight line (see Figure 1 in the text for
a typical example). We prove:

\renewcommand{\thequotethm}{\ref{t:nestedwalls}}
\begin{quotethm}[Bertram's Nested Wall Theorem]
Let $X$ be a smooth complex projective surface. Pick an ample
$\omega\in\NS(X)$ and a real class $\beta\in\NS_\R(X)$.
Then for any class $v\in\NN(X)$ representing the Chern character of a
$\mu_\omega$-semistable sheaf or torsion sheaf on $X$ there is a
half-plane $\Pi=\Plane{\beta,\omega,1}$ 
in $\stab(X)$ which contains the basic Bridgeland stability condition
$(\A_{\beta,\omega},Z_{\beta,\omega})$ such that the intersection of
the walls with $\Pi$ are nested.

In particular, if the Picard Group $\rho(X)=1$ the walls in the space
of basic stability conditions are nested.
\end{quotethm}
This theorem was conjectured by Aaron Bertram and there was already a
lot of evidence for it in the Picard rank 1 case on $\mathbb{P}^2$ and
on abelian surfaces (see \cite{MacMeachan} for the abelian surface
case). The key ingredients are the Bogomolov inequality and the Hodge Index
Theorem (just as in \cite{Arcara07}) applied to $v$.

The proof uses only the numerical constraints (we will call the
resulting walls \textbf{pseudo-walls} in what follows) and actually
proves the stronger result that the
radii of the circles depends only on the position of the centre of the
circles. It also follows from the analysis that given a Chern
character $v$ there exists a real number $C_0$ such that every wall
must cross the line $s=C_0$. This is particularly helpful to help us
locate walls by providing bounds on the Chern classes of destabilizing
objects. The number $C_0$ is given by the formula
\[\frac{c_1(v)\cdot\omega}{r(v)\omega^2}-\sqrt{F},\]
where $F$ is given by Equation \ref{e:F}. $F$ depends only on $v$ and the
choice of plane $\Pi_{\beta,\omega}$. In the special case of
Picard rank $1$, $F=\Delta(v)/(r(v)^2\omega^2)$, where $\Delta(v)$ is the
discriminant of $v$. More generally, without the condition on the
Picard rank, we show that when $F=0$ there can be no non-trivial walls at
all (see Theorem \ref{t:Fis0}). In other words, when $F=0$, a Bogomolov critical
$\mu_\omega$-semistable sheaf 
must always be $\mu_Z$-stable for any basic Bridgeland stability
condition $(\A,Z)$ in the plane $\Pi_{\beta,\omega}$.

Using our constraints on the walls and the fact that walls must correspond
to actual destabilizing objects we show that the radii of the walls are
always bounded above (and we give explicit computable
bounds). More precisely, we shall show:

\renewcommand{\thequotethm}{\ref{t:bound}}
\begin{quotethm}
Let $\beta$ be any class in $N(X)_{\mathbb{R}}$,
$\omega$ an ample class and $u\in\mathbb{R}$.
Fix a Chern character $v$ of a $\mu_\omega$-semistable sheaf or a torsion
sheaf. In the half-plane $\Plane{\beta,\omega,u}$ there are real
numbers $C_0$ and $R_0$ such that all of 
the walls corresponding to $v$ are contained in the semi-circle with centre $(C_0,0)$ and
radius $R_0$. In particular, the radii of the walls are bounded above
by $R_0$.
\end{quotethm}
Together with local finiteness of the walls we can then give
a new more direct proof of a result in \cite{LoQin} which
shows that for a given $\beta$ and $\omega$, there is some
$M\in\mathbb{R}$ such that there are no mini-walls (these are just the
intersection of walls with rays $t>0$ in our planes) for $t\in[M,\infty)$. We also
show that for all but one ray in our plane the number of walls is
actually globally finite.

We finish with a number of examples. Firstly, we show how to use $C_0$
to quickly show that, for certain Chern classes $v$, walls cannot
exist. We then look again at the
\cite{MacMeachan} situation of an irreducible principally polarized
abelian surface and rank $1$ sheaves. We also consider the same
situation but for a product abelian surface and this allows us to
construct an example to show that for other choices of plane, the
walls do not nest. We also show how to conclude global finiteness
results for walls for special choices of variety. We finish with a
result which gives an example of how to quickly deduce vanishing
results for walls. This final example looks at the 6 dimensional
O'Grady space of rank 2 charge 2 stable bundles on an abelian surface
which famously provides an exceptional compact simply-connected holomorphic
symplectic manifold. We show that there are no walls for the
associated Chern character.

\section{Stability Conditions and Their Walls}
 
\subsection{The numerics of a surface}
Let $X$ be a smooth complex projective surface and let $D(X)=D(\coh X)$ denote
its bounded derived category of coherent sheaves. We let $\NN(X)$
denote the \textbf{numerical Grothendieck group} of $X$. For a coherent sheaf
$A$ on $X$ we let $v(A)$ denote its image in $A$. This can be extended
to the derived category by defining $v(a)=\sum_{i}(-1)^iv(A^i)$, where
$a\in D(X)$ and we use the convenient abbreviated notation
$A^i=H^i(a)$.

To be more definite we shall identify $\NN(X)$ with $\Z\oplus
\NS(X)\oplus\Z[\frac{1}{2}]$, where the factors correspond to $r(A)$,
$c_1(A)$ and $\ch_2(A)[X]$. We shall freely write
$v(a)=(r(a),c_1(a),\ch_2(a)[X])$ corresponding to this
decomposition and abuse notation by writing
$\ch_2(E)=\ch_2(E)[X]$. Note that the rank of $\NS(X)$ is the
\textbf{Picard number} $\rho(X)$ of
$X$. We denote the intersection pairing on $\NS(X)$ by $(c,d)\mapsto
c\cdot d$. Its signature is $(1,\rho(X)-1)$. This is extended to a
pairing of signature $(2,\rho(X))$ on $\NN(X)$ by $(v,w)\mapsto
v_1w_3+v_3w_1+v_2\cdot w_2$.

For any $\Z$-module $M$ we write $\NN_M(X)=\NN(X)\otimes M$ and
$\NS_M(X)=\NS(X)\otimes M$. In particular, when $M=\Q$, $\R$ or
$\C$. We abuse notation and write $\alpha\in\NS_M(X)$ for the image of
some $\alpha\in NS_N(X)$, where $M$ is an $N$-module.
Suppose $\omega\in\NS(X)$ is an ample class and given any non-zero
element $\beta$ of $\NS_\R(X)$ we write $\beta=b\omega+\gamma$
for the unique element $\gamma=\gamma(\beta)$ of $\NS_\R(X)$ which is orthogonal
to $\omega$. We let $\tilde\gamma$ denote a choice of element of
$\NS(X)$ which maps to $\gamma$ in $\NS_\R(X)$. We can play the same
game in $\NS_\Q(X)$ but in that case there is a canonical choice for
$\tilde\gamma$ given by $n\gamma$ for the least positive integer
$n$. This trick is also used in Kawatani \cite{Kawat} to analyse
slightly different stability conditions.

For a fixed ample class $\omega\in\NS(X)$ we let
$\mu_\omega(v)=c_1(v)\cdot\omega/r(v)$ for any vector $v\in\NN_R(X)$,
interpreting it as $\pm\infty$ when $r(v)=0$. We let
$\mu_\omega(a)=\mu_\omega(v(a))$ for an object $a\in D(X)$. As is well known,
any sheaf $E\in\coh(X)$ has a unique Harder-Narasimhan filtration by
$\mu_\omega$-semistable sheaves $\{E_i\}_{i=1}^n$ with
$\mu_\omega(E_i)>\mu_\omega(E_{i+1})$. We write
$\mu^{+}_{\omega}(E)=\mu_{\omega}(E_1)$ and
$\mu^{-}_{\omega}(E)=\mu_{\omega}(E_n)$.
% We call $n$ the \textbf{length} of the Harder-Narasimhan filtration. 

\subsection{Distinguished subvarieties of $\stab(X)$}\label{s:dist}
Now, fix an ample class $\omega\in\NS(X)$ and another class
$\beta\in\NS_\R(X)$. Then we can
define an abelian subcategory  $\A_{\beta,\omega}$ of $D(X)$ by the
following tilt with respect to  $\coh(X)$:
\[
\begin{split} T_{\beta,\omega}=&\{E\in\Coh(X):E\text{ is torsion or
  }\mu^{-}_\omega(E)>\beta\cdot\omega\}\\ 
F_{\beta,\omega}=&\{E\in\Coh(X):E\text{ is torsion-free and }\mu^{+}_{\omega}(E)\leq\beta\cdot\omega\}
\end{split}
\]
Notice that, for any $t>0$, $\A_{\beta,t\omega}=\A_{\beta,\omega}$.
We can make these abelian subcategories into a Bridgeland stability
condition by defining
\[Z_{\beta,\omega}(v)=-\exp(-\beta-i\omega)\cdot v\]
in $\NN_\C(X)$. The fact that $\Im Z\geq0$ and equal to zero only if
$\Re Z<0$ is an easy exercise using the Hodge Index Theorem and the
Bogomolov inequality. The details are discussed in
\cite[$\S2$]{Arcara07} and also \cite[$\S4$]{BayerMacri09}. In fact, we do not seem to need the
  Harder-Narasimhan property in our analysis. We shall refer to these as \textbf{basic
  stability conditions}. To be explicit note that
\[Z_{\beta,\omega}((r,\theta,z))=-z+\theta\cdot\beta-
\frac{r}{2}(\beta^2-\omega^2)+i(\theta-r\beta)\cdot\omega.\] 

We can use this to describe subvarieties of $\stab(X)$. Firstly
$(\A_{\beta,\omega},Z_{\beta,\omega})$ is a point of $\stab(X)$
(actually, a non-trivial fact). We can also consider the set
$(\A_{\beta,\omega},Z_{\beta,t\omega})$ as a ray, denoted
$\ray{\beta,\omega}$, in $\stab(X)$. Note 
that this starts at the open boundary of $\stab(X)$ (corresponding to $t=0$).
If $\beta\neq0$, we can also consider the ``plane'' (or more properly,
``half-plane'') $\{(\A_{s\beta,\omega},Z_{s\beta,t\omega}):s\in\R,\ t>0\}$. This turns
out to be hard to describe in general and we modify this to a family
of planes as follows. Write $\beta=b\omega+\gamma$ as above and then
consider the (usually)  half $3$-space of stability conditions
\[\spc{\beta,\omega}=\{(\A_{s\omega+u\gamma,\omega},Z_{s\omega+u\gamma,t\omega}):s,u\in\R,\ t>0\}.\]
We can now consider the $u$-indexed family of planes  given as the
subspaces of $\spc{\beta,\omega}$ where $u$ is fixed. We shall
denote each of these planes by $\Plane{\beta,\omega,u}$. 
 
\begin{note}
\begin{enumerate}
\item If $\gamma=0$ (so $\beta\propto\omega$) then
  $\spc{\beta,\omega}$ is just a half-plane $\Plane{\beta,\omega,0}=\Plane{\omega,\omega,0}$.
  If $\rho(X)=1$ then $\gamma$ must be zero. 
\item The ray $\ray{0,\omega}$ is contained in   $\Plane{\omega,\omega,0}$.
\item More generally, if $b=\beta\cdot\omega=0$ but $\beta\neq0$ then
  $\spc{\beta,\omega}=\spc{\beta+\omega,\omega}$ is still three
  dimensional.
\item Suppose $\alpha\in\NS_\Q(X)$ is non-zero and orthogonal to both
  $\gamma$ and $\omega$. Then
  $Z_{\beta,\omega}(v)=Z_{\beta,\omega}(v+\alpha)$.
\end{enumerate}
\end{note}
\begin{example}
Consider the case of a product abelian surface
$(\T,\omega)=(\T_1\times\T_2,\ell_1+\ell_2)$. Assume
$\NS(\T)=\langle\ell_1,\ell_2\rangle$. Then, over $\Q$ we have
$\NS_\Q(\T)=\langle\omega,\ell_1-\ell_2\rangle$. Then
$\spc{\ell_1-\ell_2,\omega}$ contains all of the possible basic
stability conditions.
\end{example}

\subsection{Walls and pseudo-walls}
Given a stability condition $(\A,Z)$, we say that an object $a\in\A$
is \textbf{semistable} if for all proper injections $b\to a$ in $\A$ we have
$\phi(b)\leq\phi(a)$, where $\phi(b)=\arg(Z(v(b)))$. Equivalently, we
can use the slope function:
\[\begin{split}\mu_Z(b)&=-\frac{\Re Z(v(b))}{\Im Z(v(b))}\\
&=\frac{z-\theta\cdot\beta+\frac{r}{2}(\beta^2-\omega^2)}{(\theta-r\beta)\cdot\omega},
\end{split}
\]
where $v(b)=(r,\theta,z)$ and $Z=Z_{\beta,\omega}$.

Now fix a non-zero class $v\in\NN(X)$ and for the following definition
we assume that $v=v(a)$ for some object $a$ of $D(X)$.
\begin{dfn}
We say that a class $w\in\NN(X)\setminus\langle v\rangle$ is
\textbf{critical} for $v$ if each of the following condition hold:
\begin{quote} There is a $\sigma\in\stab(X)$ and objects $a,b\in\A_\sigma$
  with an injection $b\to a$ in $\A_\sigma$  such that $v(a)=v$ and
  $v(b)=w$ and $\mu_Z(a)=\mu_Z(b)$. 
\end{quote}
If $w$ is critical for $v$ then we write $W^v_w\subset\stab(X)$ for
the set of such $\sigma$ and call it the \textbf{wall} corresponding
to $w$. We drop the $v$ from the notation if it is understood.
\end{dfn}

\begin{note}
If there is no $a\in D(X)$ such that $v(a)=v$ then there can be no
critical $w$. Generally, we assume that $v$ is the Chern character of
some object.
\end{note}
If a wall exists then it must be a real codimension 1 subspace of
$\stab(X)$. The complement of the union of all of the walls is a
disjoint union of chambers with the property that if for any $\sigma_0$
in a chamber and object $a$ is $\sigma_0$-stable, then it remains
$\sigma$-stable for all $\sigma$ in the same chamber. One would hope
that there is a good coarse moduli scheme for such objects and
crossing a wall would correspond to a surgery on these moduli
spaces. We do not pursue the details here but instead use this as
motivation to provide information about finding such walls. To help us
do this we introduce a weaker notion of \textbf{pseudo-wall}.
\begin{dfn}
We say that a class $w\in\NN(X)\setminus\langle v\rangle$
(respectively, $w\in\NN_\R(X)\setminus\langle v\rangle$) is
\textbf{$\Z$-critical} (respectively, \textbf{$\R$-critical}) for $v$ if
there is a stability condition $\sigma$ such that $\mu_Z(w)=\mu_Z(v)$.
\end{dfn}

\begin{remark}
It is clear that a wall is also an integral pseudo-wall and an
integral pseudo-wall is also a real pseudo-wall. But the converses are
generally false.
\end{remark}

\begin{remark}
The set of pseudo-walls (real or integral) do not depend very much on $X$.
\end{remark}

\begin{remark}\label{r:constrainv}
Note that we will be interested in these purely numerical constraints
on the existence of walls (embodied in the pseudo-walls) which are
well defined for any $v\in\NN_M(X)$ but in the applications we will 
always want $v$ to be a valid class corresponding to either a
$\mu_\omega$-semistable sheaf $E$ or a torsion sheaf $E$. The Bogomolov
inequality then implies that $c_1(v)^2\geq 2r(v)\ch_2(v)$. Note also
that if $v=v(E)$ then $E\in T_{s\omega+u\gamma,\omega}$ and so either
$r(E)=0$ or $\mu_\omega(E)>s\omega^2$.
\end{remark}

\begin{dfn}
We say that $v\in\NN_M(X)$ (for any $M$) is a \textbf{Bogomolov class}
if it satisfies the Bogomolov inequality:
\[c_1(v)^2\geq 2r(v)\ch_2(v).\]
\end{dfn}

\begin{remark}\label{r:twist}
Finally, in this section observe that twisting $v$ by $\omega$ shifts the
walls and pseudo-walls by  a unit to the right. 
\end{remark}

\section{Purely Numerical Constraints}
Now fix a Bogomolov class $v=(x,\theta,z)$  and an ample class $\omega$ in
$\NS(X)$. Assume for the moment that $\rho(X)>1$ and fix a class
$\gamma\in\NS_{\Q}(X)$ which is orthogonal to $\omega$. Write
$\NS_{\Q}(X)=\langle\omega\rangle\oplus\langle\gamma\rangle\oplus\langle\omega,\gamma\rangle^\perp$. 
Decompose $\theta=y_1\omega+y_2\gamma+\alpha$, where
$\alpha\in\langle\omega,\gamma\rangle^\perp$. If $\rho(X)=1$ then we
can usually just let $\gamma=0$. Let $Z=Z_{s\omega+u\gamma,t\omega}$. 
We shall also abbreviate $\omega^2=g$ and $\gamma^2=-d$. Note that
$d\geq0$ by the Hodge Index Theorem and $d=0$ if and only if $\gamma=0$.
Then 
\[\mu_Z(v)=\frac{z-sy_1g+uy_2d+\dfrac{x}{2}(s^2g-u^2d-t^2g)}{(y_1-xs)gt}.\]
Let $w=(r,c_1\omega+c_2\gamma+\alpha',\chi)$. Then 
\[
\mu_Z(w)-\mu_Z(v)=\frac{\begin{aligned}(y_1-xs)\Bigl(\chi-&sc_1g+uc_2d+
\frac{r}{2}(s^2g-u^2d-t^2g)\Bigr)\\&-(c_1-rs)\Bigl(z-sy_1g+uy_2d+\frac{x}{2}(s^2g-u^2d-t^2g)\Bigr)
\end{aligned}
}{(c_1-rs)(y_1-xs)gt}
\]
The numerator$\times2$ then equals (collecting terms in $s$ and $t$):
\begin{multline*}
g\left(xc_1-y_{{1}}r\right) {s}^{2}-
 2\left(\chi x-rz+c_{{2}}udx-y_{{2}}udr\right) s+ \\
g\left(xc_1-y
_{{1}}r\right) {t}^{2}-2zc_{{1}}+2c_{{2}}u
dy_{{1}}+x{u}^{2}dc_{{1}}-2y_{{2}}udc_{{1}}+2\chi y_{{1}}-r{u}^{2}dy_{{1}}.
\end{multline*}
Now assume $\mu_\omega(v)\neq\mu_\omega(w)$ (so we can divide by $c_1x-ry_1$).
Completing the squares we can write this as
\[g(xc_1-ry_1)\left((s-C)^2+t^2-D-C^2\right),\]
where
\[\begin{split}
C&=\frac{x\chi-rz +ud(xc_2-ry_2)}{g(xc_1-ry_1)}\\
D=&\frac {2zc_1-2c_{2}udy_{1}-xu^{2}dc_1+2y_{2}udc_{1}-2\chi y_{1}+r
    u^{2}dy_{1}}{g(xc_1-ry_1) }
\end{split}
\]
This shows that in the plane $\Plane{\gamma,\omega,u}$ using
  coordinates $(s,t)$, the walls are semicircles with centres $(C,0)$
  and radii $R=\sqrt{D+C^2}$. 

\begin{notation}
We will generally use $C$ and $R$ to denote the centre and radius of
any pseudo-wall in question. To make the parameters clearer we will
occasionally write $C=C_v(w)$ and $R=R_v(w)$ emphasizing the fact that
these are functions of a variable $w$ indexed by a parameter $v$.
\end{notation}

\begin{remark}\label{r:sameslope}
In the case where $\mu_\omega(w)=\mu_\omega(v)$ there are no square
terms and the wall is a line of the form $s=y_1/x$ which, for
completeness, we can view as a circle of infinite radius.
\end{remark}

We now want to show that the dependence of the radii on $w$ is
entirely through the centres.
\begin{lemma}\label{l:radiuscentre}
With the notation as above, if $x\neq0$,
\[D=\frac{ud(2y_2-ux)+2z}{gx}-\frac{2y_1}{x}C.
\]
\end{lemma}
\begin{proof}
Simplify $D+\dfrac{2y_1}{x}C$. The details are left to the
reader.
\end{proof}
When $x=0$ we do not have to consider the radii:
\begin{lemma}\label{l:zerorank}
If $x=0$ and $y_1>0$ then $C=\dfrac{z+duy_2}{gy_1}$. In particular, the centre is
independent of $w$.
\end{lemma}
\noindent We leave the elementary proof as an exercise for the reader. 

\begin{remark}\label{r:bogom}
So, if $x>0$,  the radius is
\[
R=\sqrt{D+C^2}
=\sqrt{\left(C-\frac{y_1}{x}\right)^2-F},
\]
where 
\begin{equation}\label{e:F}
F=\frac{d}{g}\left(u-\frac{y_2}{x}\right)^2+\frac{1}{x^2g}(y_1^2g-y_2^2d-2xz).
\end{equation}
When $v=v(E)$, for some sheaf $E$, the last term in $F$ is $\frac{1}{x^2g}$ times
$\Delta(E)-\alpha^2$, where $\Delta(E)$ is the discriminant of $E$. If we assume
$v$ is a Bogomolov class then this is non-negative ($\alpha^2\leq0$ by
the Hodge Index Theorem). Since, $d\geq0$ by the Hodge
Index Theorem we have $F\geq0$ for all $u$.
\end{remark}

\subsection{Real constraints}
We can now prove:
\begin{prop}\label{p:nestedwalls}
Suppose $v$ is a Bogomolov class. Then
in $\Plane{\gamma,\omega,u}$, the real pseudo-walls are
\textbf{nested}. By nested we mean that if $W_1$ and $W_2$ are two
distinct (semi-circular) pseudo-walls in $\Plane{\gamma,\omega,u}$ then
either $W_1$ is entirely contained in the interior of the semi-circle
$W_2$ or vice-versa. In particular, distinct pseudo-walls cannot intersect
and one pseudo-wall cannot be entirely outside any other pseudo-wall.
\end{prop}

\begin{proof}
When $x=0$ this follows immediately from Lemma \ref{l:zerorank} as
the centres of the circles are fixed. Note that $y_1>0$ unless $v$ is
the class of a sheaf supported in dimension $0$. So we suppose that
$x\neq0$. The result will follow from Lemma \ref{l:radiuscentre} if
we can show that $C+R$ and $R-C$ are both monotonic
increasing (or decreasing) as functions of $C$. Note that
\[\frac{d}{dC}(C\pm R)=1\pm \frac{1}{R}\left(C-\frac{y_1}{x}\right).
\]
But $F\geq0$ for all $u$ by the Remark \ref{r:bogom}. Hence,
$R\leq|C-y_1/x|$ and so if $C\neq y_1/x$, 
$\left|\dfrac{d}{dC}(C+R)\right|\geq0$ with equality if and only if
$F=0$.
Similarly $\left|\dfrac{d}{dC}(C-R)\right|\geq0$ for all $C\neq
y_1/x$ but has the opposite sign to the $C+R$ case. If $F=0$ then
\[\{C+R,C-R\}=\left\{2C-\dfrac{y_1}{x},\ \dfrac{y_1}{x}\right\}\]
and again we are done.
\end{proof}
Note that two pseudo-walls may be coincident. In fact, there will typically
be continuous families of $w$ giving coincident pseudo-walls (by solving
$C=C_0$ for $w$).

\begin{remark}\label{r:upperbound}
 Now $R>0$ exactly when
\[\left(C-\frac{y_1}{x}\right)^2>F\geq0.\]
So either $C>y_1/x+\sqrt{F}$ or $C<y_1/x-\sqrt{F}$. But, by
Remark \ref{r:constrainv}, we also have $y_1/x>s$ and so $C<y_1/x$.
Then $\dfrac{d}{dC}(C-R)\geq2$. Note also that $R\to\infty$ as
$C\to-\infty$. Consequently, we have that any point
$(s,t)\in\{(a,b):a<y_1/x,\ b>0\}$ is contained in at most one real
pseudo-wall (see Figure 1).
\end{remark}

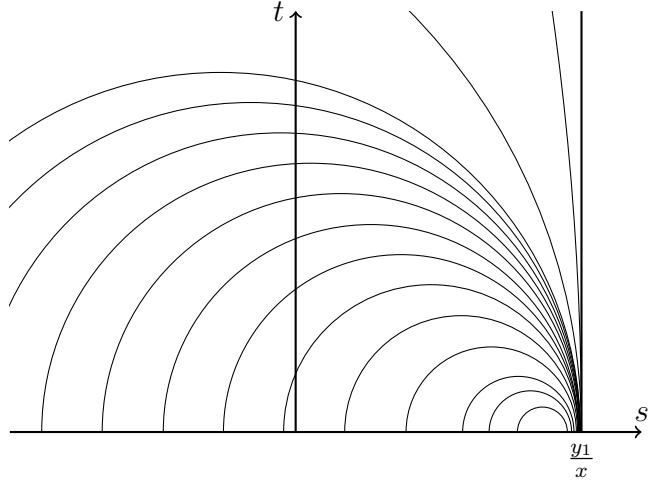
\begin{figure}
\begin{tikzpicture}[scale=0.4]
\draw[->,thick] (-20,0) --(1,0) node[above] {$s$};
\draw[->,thick] (-10.5,0) -- (-10.5,14) node[left] {$t$};
\draw[-,thick] (-1,0) node [below] {$\frac{y_1}{x}$}--(-1,14) ;
\begin{scope}
\clip (-20,0) rectangle (1.2,14);
\foreach \x/\y in {1.3/0.83,1.7/1.37,2.1/1.85,3/2.83,4/3.87,5/4.90,6/5.9,7/6.9,8/7.93,9/8.94,10/9.95,11/10.96,12/11.96,20/20,100/100} {
\draw (-1cm-\x cm,0cm) circle (\y cm);
}
\end{scope}
\end{tikzpicture}
\caption{Typical real pseudo-walls}
\end{figure}
\begin{cor}
For $F>0$, any pseudo-wall must intersect $s=C_0$, where
\[C_0=\begin{cases} \dfrac{y_1}{x}-\sqrt{F}&\text{ if $x>0$.}\\
\dfrac{z+duy_2}{gy_1}&\text{ if $x=0$ and $y_1>0$.}
\end{cases}
\]
\end{cor}
\begin{remark}\label{r:cbound}
  This is particularly useful when trying to locate walls. For
  example, we must have
\[0<c_1-rs<y_1-xs\]
for the characteristic classes of a destabilizing object of an object
$e\in\A_s$. Since any such object must destabilize for some $t$ when
$s=C_0$ we obtain the following useful bounds on $c_1$:
  \[rC_0<c_1<y_1+(r-x)C_0.\]
See Example \ref{eg:rankzero} and Proposition \ref{p:ogrady} below.
\end{remark}

\begin{remark}
Using the result of Lo and Qin \cite{LoQin} which says that for each
ray there is an upper bound on the radius we can now deduce that
bound is uniform in $s$. In the next section we will give an explicit
such bound but it will also follow that the bound is not uniform in
$u$ (and we give an example in a later section to show that it can be
unbounded). 
\end{remark}

\section{Applications}
\subsection{Bertram's Nested Wall Theorem}
We can use these technical results to quickly prove:
\begin{thm}\label{t:nestedwalls}
Let $X$ be a smooth complex projective surface. Pick an ample
$\omega\in\NS(X)$ and a real class $\beta\in\NS_\R(X)$.
Then for any class $v\in\NN(X)$ representing the Chern character of a
$\mu_\omega$-semistable sheaf or a torsion sheaf on $X$ there is a half-plane $\Pi=\Plane{\beta,\omega,1}$
in $\stab(X)$ which contains the basic Bridgeland stability condition
$(\A_{\beta,\omega},Z_{\beta,\omega})$ such that the intersection of
the walls with $\Pi$ are nested.

In particular, if $\rho(X)=1$ the walls in the space of basic stability
conditions are nested.
\end{thm}

\begin{proof}
We decompose $\beta=b\omega+\gamma$ as usual with
$\gamma\cdot\omega=0$. Then Proposition \ref{p:nestedwalls} implies
that the real pseudo-walls are nested. Since the actual walls are a
subset of these walls, these must be nested as well.
\end{proof}

One immediate corollary is that each chamber is path connected to the
$t=0$ axis:
\begin{cor}
For any $\beta$ and $\omega$, the moduli space of S-equivalence
classes of $\mu_Z$-semistable objects with Chern character $v$ for
$Z=Z_{s\omega+u\beta,t\omega}$ is isomorphic to a moduli space of
$\mu_Z$-semistable objects with the same Chern character $v$ for some $t$ arbitrarily small.
\end{cor}
This shows that Bertram's ``sliding down the wall'' trick used in
\cite{MacMeachan} always works.
In the case of abelian surfaces and K3 surfaces this implies that such
moduli spaces are projective varieties and represent the appropriate
moduli functor (see \cite{MacMeachan} and \cite{MYY11b} for special
cases of this).

\subsection{A lower bound on the centres}
We have seen in Remark \ref{r:upperbound} that there is an upper bound
for the centres of the walls. We now show that the centres of
actual walls are also bounded below as well (at least for a large
class of $v$). This follows theoretically from the work of Lo and Qin
(\cite[Theorem 1.1(ii)]{LoQin}) who show that in each ray there are no
walls for t sufficiently large. Their construction is not constructive
and our aim here is to use the Nested Wall Theorem to compute explicit bounds.

We prove this in a series of lemmas. For these technical lemmas we
need some parameters from the surface. Let $\tilde
g=\gcd_{E}(c_1(E)\cdot\omega)$ over all sheaves $E$. Then $g/\tilde
g=g'$ is an integer and $\tilde y_1=y_1g'$ is always an integer (even though
$y_1$ may not be).  Initially we shall assume that $x=r(v)\neq0$ and deal with the
$x=0$ case afterwards. For any such $v$ we let
$p=\gcd(r(v),c_1(v)\cdot\omega/\tilde g)=\gcd(x,\tilde y_1)$ in the notation
of the previous section. What we do in the lemma is to show that if
we choose $s$ sufficiently close to $y_1/x$ we can ensure that there
are no subsheaves of $E$ in $A_s$. We have to work harder to show that
there can be no $\mu_Z$-destabilizing objects of rank bigger than
$r(E)$. When $p=1$ it turns out that the same bound works but we have
to go closer to $y_1/x$ if $p>1$. 

The bound is constructed as follows. Using Euclid's algorithm,
consider the set of pairs of integers $(m,n)$ such that
$n\tilde y_1-mx=p$. We assume $n>0$. Fix some $(m_0,n_0)$ and then observe that
any other pair is given by $(m,n)=(m_0+y'\lambda,n_0+x'\lambda)$ for some
integer $\lambda$, where $y'=\tilde y_1/p$ and $x'=x/p$. Note that the
ratio $m/(g'n)$ tends to $y_1/x$ from below as
$n\to \infty$ (whatever sign $y_1$ takes). So given any positive real
number $\xi$, we can pick some $\lambda$ such that 
 the set
\[\left\{(j,k)\in\mathbb{Z}^2:\dfrac{m}{n}=\dfrac{m_0+\lambda y'}{n_0+\lambda
    x'}<\frac{j}{k}<\frac{\tilde y_1}{x},\
\gcd(j,k)=1,\ 0<k\leq \xi x\right\}\]
is empty. Let us, for the moment, denote the least choice of such $m$
and $n$ for a given $\xi$ by $m_0$ and $n_0$ and define
the set 
\[P_\xi=\{(m,n):n\tilde y_1-mx=p\text{ and }|m|\geq |m_0|,\ n\geq
n_0,\ n>\dfrac{p}{g'}\sqrt{\frac{g}{\Delta}}\},\]
where $\Delta=gy_1^2-dy_2^2-\alpha^2$. This last mysterious condition
on $n$ often imposes no constraint but it will ensure that any wall
which is sufficiently near $y_1/x$ must also cross $m/(g'n)$ whenever
$(m,n)\in P_\xi$.

Now suppose $E$ is a torsion-free sheaf with $v(E)=v$ and $s<y_1/x$. Then
$\mu_\omega(E)=\dfrac{m\tilde g}{n}+\dfrac{\tilde
  g}{nx'}=\dfrac{\tilde y_1\tilde g}{x}=\dfrac{y_1g}{x}>sg$
and so it satisfies a necessary condition for $E\in\A_{s\omega+u\gamma,\omega}$.

Our strategy is as follows:
\begin{enumerate}
\item  We show that if a wall does cross $s=m/(g'n)$ and some other
  predetermined line closer to $s=y_1/x$ then the rank of any
  destabilizer must be at most $\xi x$ for a suitable $\xi$ depending
  only on $v$ and $X$. For $p=1$ we do not require this extra
  assumption and we can set $\xi=1$.
\item Assuming $F>0$, we show for any $(m,n)\in P_\xi$ for the $\xi$ given in (1) then
  any object which is $\mu_Z$-stable for any $Z$ in the region
  $(m/(g'n),y_1/x)\times\R_{>0}$ must be a $\mu_\omega$-semistable
  sheaf. 
\item We can then deduce that object must be $\mu_Z$-stable for all
  $Z$ in the region.
\item For $F=0$ we show directly that there are no walls at all.
\item We also need to assume $x>0$ but we will finish this section by
 showing that $y_1$ is always an upper bound for $R$ when $x=0$.
\end{enumerate}

First of all, set $\xi=1$.
\begin{lemma}\label{l:lowrank}
With this notation, suppose $s\geq m_0/(g'n_0)$. If
$E\in\A_{s\omega+u\gamma,\omega}=\A_s$ is a torsion-free sheaf which
is $Z_{s\omega+u\gamma,t\omega}$-stable for some $s<y_1/x$ and some $t>0$ then $E$ must be
$\mu_\omega$-semistable and, for any $t>0$, each
$Z_{s\omega+u\gamma,t\omega}$-destabilizing subobject of $E$ in $\A_s$ must
have rank strictly bigger than $r(E)$.
\end{lemma}

\begin{proof}
Since
$\dfrac{\tilde y_1\tilde g}{x}\geq\mu^-_\omega(E)>sg>\dfrac{\tilde g m_0}{n_0}$ and 
if $E$ is not $\mu_\omega$-semistable then $\mu^{-}_\omega(E)$ is
given by $\mu(F)$ for some sheaf $F$ with 
$r(F)<r(E)$, which contradicts the definition of $m_0$ and $n_0$. So $E$
must be $\mu_\omega$-semistable.

Suppose that $k\to E$  $Z_{s\omega+u\gamma,t\omega}$-destabilizes $E$ in
$\A_s$ with quotient $q$ and assume that $r(K)\leq x$. Note that $k$ must be a sheaf
$K$. We have a short exact sequence $K/Q^{-1}\to E\to Q^0$. Suppose
$Q^0_{\free}\neq0$ then since
$c_1(Q^0_{\free})\cdot\omega\leq c_1(Q^0)\cdot\omega$ and both
$K/Q^{-1}$ and $Q^0$ are in $T$, we must have one of them with slope
less than or equal to the slope of $E$. But then $K/Q^{-1}$
$Z_{s\omega+u\gamma,t\omega}$-destabilizes $E$ for all $s<y_1/x$ and
$t>0$ by Remark \ref{r:sameslope}. So
$Q^{0}$ must be a torsion sheaf. On the other hand,
$c_1(Q^0)\cdot\omega=0$ as otherwise,
$sg<\mu_\omega(K/Q^{-1})<\mu_\omega(E)$ which is impossible by
assumption. So $Q^0$ is supported on points.  But then $Q^{-1}$ must
be non-zero and so $r(K)>r(E)$.
\end{proof}
Since $m_0/(g'n_0)$ increases with $\xi$, this lemma also holds for
any $\xi>1$ as well.

For the rank bigger than $x$ case we appeal to a different
lemma. Observe first that, for any $\xi$ and $(m,n)\in P_\xi$ we have
\[\frac{m}{g'n}+\frac{2\tilde g p-1}{2gnx}=\frac{y_1}{x}-\frac{1}{2gnx}.\]
What we do now is to show that the critical wall for a destabilizing
object cannot lie outside the wall given by $s=m/(g'n)$ unless the rank is at most $r(v)$.

\begin{lemma}\label{l:highrank}
Pick any $\xi>0$ and $(m,n)\in P_\xi$. Let $s=m/(g'n)$.
Suppose $p=1$ and $e\in \A_s$ satisfies
$v(e)=v$. Then $e$ has no subobjects $k$ in $\A_s$ such that both
$\mu_Z(k)$ and $\mu_Z(e/k)$ are finite.
\end{lemma}

\begin{proof}
Suppose $k\subset e$ in $\A_s$ such that $\mu_Z(k)$ and $\mu_Z(e/k)$
are finite.
Let $\ch(k)=(r,c_1\omega+c_2\gamma+\alpha',\chi)$.
Since, $Z$ is a
stability condition and we have $0<\Im Z(k)<\Im Z(e)$ (from the
assumptions on $k$ and $e/k$).
Then
\begin{equation}\label{e:pisone}
0<c_1-r\frac{m}{g'n}<y_1-x\frac{m}{g'n}=\frac{n\tilde y_1-mx}{g'n}=\frac{p}{g'n}.
\end{equation}
Clearing denominators and setting $p=1$, we have $0<c_1g'n-rm<1$. But $c_1g'n-rm$ is an
integer. 
\end{proof}
In particular, this means that, for $p=1$, no wall of finite radius can intersect
$s=m/(g'n)$. This is because, if $\mu_Z(e/k)$ is infinite then $k$ cannot
$Z$-destabilize. If $\mu_Z(k)$ is infinite then $k$ must destabilize
for all $t>0$ which means that $s=m/(g'n)$ must be the wall. 

\begin{lemma}\label{l:highrankp}
Suppose $p>1$. Let $(m,n)\in P_\xi$ for $\xi=\dfrac{2\tilde g
  p}{2\tilde g p-1}$ 
and let $s\geq \dfrac{m_0}{g'n_0}+\dfrac{2\tilde g
  p-1}{2gn_0x}$.
Suppose $e\in \A_s$ satisfies
$v(e)=v$. Suppose $e$ has a subobject $k$ in $\A_s$ such that both
$\mu_Z(k)$ and $\mu_Z(e/k)$ are finite. Then $0\leq r(k)\leq r(e)$ and
$0\leq r(q)\leq r(e)$.
\end{lemma}
\begin{proof}
If the conclusion does not hold then either $r(q)$ or $r(k)$ is bigger
than $x=r(e)$. Without loss of generality assume $r(k)>r(e)$ and set
$c_1=c_1(k)\cdot\omega/g$ and $r=r(k)$. Since $c_1-rs<y_1-xs$, we have that
$c_1<y_1+(r-x)s$ and, since $s<y_1/x$, this gives
$c_1<y_1+(r-x)y_1/x=ry_1/x$. In other words,
$\mu_{\omega}(k)<\mu_{\omega}(e)$. 

But now
\[0<c_1-rs\leq c_1-r\left(\dfrac{m_0}{g'n_0}+\dfrac{2\tilde g
  p-1}{2gn_0x}\right).\] 
Rearranging gives 
\[r\dfrac{2\tilde g p-1}{2\tilde gx}<c_1g'n_0-m_0r.\]
Combining with (\ref{e:pisone}) gives
\[r<\frac{2\tilde g p}{2\tilde g p-1} x=\xi x.\]
But this is impossible for our choice of $(m_0,n_0)$. 
\end{proof}

We can now treat the $F=0$ case.
\begin{thm}\label{t:Fis0}
If $F=0$ then there are no walls in $\Pi_{\beta,\omega,y_2/x}$ except
when $\beta\cdot\omega=y_1g/x$. 
\end{thm}
\begin{proof}
When $F=0$, we have $C+R=y_1/x$ for any wall. Consequently, for any
$\xi>0$, any such wall must intersect some $s=\beta\cdot\omega/g=m/(g'n)$ for some
$(m,n)\in P_\xi$. When $p=1$ this contradicts Lemma \ref{l:highrank}
(by the remark following the lemma. It must also intersect
$s=\dfrac{y_1}{x}-\dfrac{1}{2gnx}$ and so when $p>1$, Lemma
\ref{l:highrankp} tells us that any object $e$ (with $v(e)=v$) in $\A_s$
must be a sheaf since $E^{-1}[1]$ is a subobject and
$r(E^{-1}[1])<0$ if $E^{-1}$ is non-zero. On the other hand, if a
sheaf $E$ is $\mu_Z$-stable for all sufficiently large $t$ (if it
corresponds to a wall, for example) then its torsion subsheaf must
eventually $Z$-destabilize it (for all large enough values of $t$)
because the $z$-slope of a torsion sheaf is 
$1/t$ times a (finite) constant in $t$. So $E$ must be torsion-free. Then, by
Lemma \ref{l:lowrank}, the rank of a $Z$-destabilizing subsheaf must
be bigger than $r(E)$. But that contradicts Lemma \ref{l:highrankp}
again. In any case, for $p\geq 1$, there cannot be such a wall.
\end{proof}
\begin{note}
  \begin{enumerate}
  \item The condition $F=0$ is a actually rather strong. It implies
    that $u=y_2/x$ (or $\gamma=0$) and the discriminant $\Delta(E)=\alpha^2$. Then the
    Hodge Index Theorem implies that $\alpha=0$ 
    and $\Delta(E)=0$. So these are Bogomolov critical sheaves.
\item Note that for line bundles $L$ on any smooth projective surface
  with $\alpha=0$  automatically satisfy $F=0$ on the plane given by
  $u=y_2$ and so we conclude they are Bridgeland stable there. This
  is already well known (see \cite[Prop 3.6(b)]{Arcara09}) in the case
  where $y_2=0$. 
\item When $\rho(X)=1$ and $K_X=0$ then $F=0$ on a K3 surface only for line
  bundles but for abelian surfaces this happens
  exactly for semi-homogeneous bundles and we conclude that these are
  always $\mu_Z$-stable for all $Z$. This gives a proof independently
  of Fourier-Mukai transform methods.
\item It should be possible to repeat the analysis and these theorem
  for the refined basic stability condition given for K3 surfaces by
  Bridgeland (\cite{BrK3} and used in \cite{Arcara07} and
  \cite{MYY11a}) in which $v$ is multiplied by $\sqrt{\td(X)}$. In
  that case the critical objects should correspond to spherical objects.
\end{enumerate}
\end{note}

We can now assume $F>0$. Note that if $(m,n)\in P_\xi$ for any $\xi>0$
then we require 
\begin{equation}\label{e:ncond}
n>\dfrac{p}{g'}\sqrt{\dfrac{g}{\Delta}}.
\end{equation}
 We now use this condition
by observing that $F\geq\dfrac{\Delta}{x^2g}>\dfrac{p^2}{{g'}^2n^2x^2}$
for all $u$ and so
\[C<\dfrac{y_1}{x}-\sqrt{F}<\dfrac{y_1}{x}-\dfrac{p}{g'nx}=\dfrac{m}{g'n}.\]
So any wall which crosses any line $s=s_0>m/(g'n)$ for any $(m,n)\in P_\xi$
and $\xi\geq 1$ must also cross $s=m/(g'n)$ as well. In particular, this
holds for the least such $n$.

Now let $\xi_p=1$ if $p=1$ and $\xi_p=\dfrac{2\tilde g p}{2\tilde g p-1}$ if $p>1$.

\begin{lemma}\label{l:zstabismustab}
Let $s_0=m/(g'n)$ for some $(m,n)\in P_\xi$ for $\xi=\xi_p$.
Suppose $e$ is an object of $\A_s$ for some $s_0<s<y_1/x$, where $v(e)=v$ and 
$r(v)>0$. Suppose for (at least) one $s$ there is a real number $t>0$ such
that $e$ is $\mu_Z$-stable with $Z=Z_{s\omega+u\gamma,t\omega}$. Then $e$ must be a
$\mu_\omega$-semistable sheaf.
\end{lemma}

\begin{proof}
The argument is the same as in the proof of Theorem \ref{t:Fis0}.
\end{proof}

So what we have shown is that there is some $s_0<y_1/x$ (which can be
computed explicitly) with the following property. If $e$ is an object
of some $\A_s$ for $s_0<s<y_1/x$ and is $\mu_Z$-stable for some $t$
then it must be a $\mu_\omega$-semistable sheaf by
\ref{l:zstabismustab}. But then by lemmas \ref{l:lowrank},
\ref{l:highrank} and \ref{l:highrankp} it must be $\mu_Z$-stable for
all $s_0<s<y_1/x$ and $t>0$.
What this means is that no wall can intersect either the region
$s\geq \dfrac{m}{g'n}$ (when $p=1$) or $s\geq\dfrac{2gny_1-1}{2gnx}$ (when
$p>1$). For the moment combine these to give the bound
\[\dfrac{y_1}{x}-\dfrac{(2p\tilde g-1)\epsilon+1}{2gnx},\]
where $\epsilon=1$ when $p=1$ and $=0$ otherwise. Then
$C+R\leq \dfrac{y_1}{x}-\dfrac{(2p\tilde g-1)\epsilon+1}{2gnx}$.
% ie
%\[\left(C-\dfrac{y_1}{x}-\dfrac{(2p\tilde g-1)\epsilon+1}{2gnx}\right)^2\geq
%R^2=\left(C-\frac{y_1}{x}\right)^2-F.\]
Solving for $C$ we obtain:
\begin{prop}\label{p:rankplus}
Suppose $v$ is the Chern character of a $\mu_\omega$-semistable
sheaf. Then the centres $C$ of the walls in $\Plane{\beta,\omega,u}$ are bounded
below by 
\[\frac{1}{2}\left(\frac{mx+n\tilde y_1}{g'nx}-Fg'nx\right)\text{ when
  $p=1$ or}\]
\[\frac{4gny_1-1}{4gnx}-Fgnx\text{ when $p>1$.}\]
\end{prop}

When $x=0$ we proceed differently. Recall from Lemma \ref{l:zerorank}
that all walls have a fixed centre:
\[C=\dfrac{z+duy_2}{gy_1}.\]
Observe that if $w=(r,(c_1\omega+c_2\gamma+\alpha'),\chi)$ corresponds
to a wall then $r>0$ because then $\mu_Z(w)-\mu_Z(v)$ would be
constant. But using Remark \ref{r:cbound}, we have
\[C<\frac{c_1}{r}<\frac{y_1}{r}+C\leq y_1+C.\]
In other words, the slope of $w$ is bounded by constants.
But, an object $e\in\A_s$ with $\ch(e)=w$ must satisfy
$\mu_\omega(e)>s$ and so if the wall has radius $R=R_v(w)$, we have
$C+R<\dfrac{c_1}{r}<y_1+C$ and so $R<y_1$. So we have proved:
\begin{prop}\label{p:zerorank}
If $x=0$ then the radius of any wall for
$v=(0,y_1\omega+y_2\gamma+\alpha,z)$ is less than $y_1$.
\end{prop}

Combining the above results with Proposition \ref{p:nestedwalls} we
have the main theorem of the paper: 
\begin{thm}\label{t:bound}
Let $\beta$ be any class in $N(X)_{\mathbb{R}}$,
$\omega$ an ample class and $u\in\mathbb{R}$.
Fix a Chern character $v$ of a $\mu_\omega$-semistable sheaf or a torsion
sheaf. In the half-plane $\Plane{\beta,\omega,u}$ there are real
numbers $C_0$ and $R_0$ such that all of 
the walls corresponding to $v$ are contained in the semi-circle with centre $(C_0,0)$ and
radius $R_0$. In particular, the radii of the walls are bounded above
by $R_0$.
\end{thm}

\begin{note}\label{n:summary}
The bounds are not always sharp. This is especially true for
  $p>1$ where the $p=1$ bound may often suffice. But even when $p=1$,
  we shall see an example below where the bound is not quite sharpest.
The bounds are constructive (we shall give
    explicit examples in the next section). For the first three points of
    this note we assume $p=\gcd(x,\tilde y_1)=1$.
  \begin{enumerate}
  \item For $(m,n)\in P_1$ with $n$ least
\[\frac{1}{2}\left(\frac{mx+n\tilde y_1}{g'nx}-Fg'nx\right)\leq
C<\frac{y_1}{x}-\sqrt{F},\]
and
\[0<R\leq
\frac{1}{2}\left(Fg'nx-\frac{1}{g'nx}\right),\]
where $F$ is given by Equation \ref{e:F}
and $g'=g/\gcd_E(c_1(E)\cdot\omega)$.
\item The upper and lower bounds on $s$ are
\[\frac{y_1}{x}-Fg'nx\leq s<\frac{m}{g'n}.\]
\item Recall the construction of $m$ and $n$: they are given by
  Euclid's algorithm applied to $n\tilde y_1-mx=1$ and such
  that
there are no integers $j$ and $0<k\leq x$ such that
\[\frac{m}{n}<\frac{j}{k}<\frac{\tilde y_1}{x}.\]
We also had the technical assumption on $n$ given by \ref{e:ncond} but
this would rarely need to be imposed as the RHS is often less than $1$.
It also makes sense to choose $m$ to be the least such.
For example (ignoring \ref{e:ncond}), suppose $g'=1$ then if $y_1=2$
and $x=1$, we can pick $m=1=n$. Whereas, 
if $y_1=10$ and $x=5$ then we pick $m=7$ and $n=4$.
\item The dependence on $u$ is quadratic entirely through $F$. In
  particular, it is potentially unbounded as $u$ increases. We shall
  see an example below where the walls are actually unbounded in
  the $(u,s)$-plane.
\end{enumerate}
\end{note}

Combining the theorem with local finiteness of the walls we can state
a version of Theorem 1.1(i) in \cite{LoQin}:
\begin{thm}\label{t:finite}
Suppose $F>0$ for some $u$. Then, for an $s\neq
\dfrac{y_1}{x}-\sqrt{F}$, there are only a finite number 
of mini-walls in the ray $\ray{s\omega+\gamma,\omega}$.
\end{thm}

\section{Examples and Counterexamples}

\subsection{Rank $0$ case}
\begin{prop}\label{eg:rankzero}
Suppose $(X,\omega)$ is a surface with $\rho(X)=1$ and $g'=1$. Then there are
no walls in the basic Bridgeland stability plane for $v=(0,1,kg)$ for
any integer $k$.
\end{prop}
\begin{proof}
Let $E$ be a torsion sheaf with
$\ch(E)=v$ (a twist of the structure sheaf of a divisor in the linear
system of $\omega$) and suppose $K\to E$ is a (maximally)
destabilizing object in $\A_s$ 
for some $s$ with $\ch(K)=(r,c,\chi)$. Observe that $K$ cannot have
torsion because if it did 
$K_{\text{tors}}$ would have to have
$c=1$ and so $E/K_\text{tors}$ would be a skyscraper and so
$\mu_Z(K)<\mu_K(K_{\free})$ contradicting the maximality. So (at
least) $r>0$. It also follows by a similar argument that $K$ is
$\mu_\omega$-semistable. We use 
Remark \ref{r:cbound}. We have $C_0=z/g=k$ and so we have
\[kr<c<1+kr,\]
which is a contradiction.
\end{proof}
When $z\neq kg$ then there may be walls. For example, consider
$E=\OO(\omega)/\OO_X$ for some map $\OO_X\to\OO(\omega)$. This is just
a line bundle supported on a divisor in the linear system $|\omega|$.
This has Chern character $(0,\omega,g/2)$ and there is
a wall with centre $1/2$ and radius $1/2$ given by $K=\OO(\omega)$.

\subsection{Rank $1$ case}
We now consider the cases studied in \cite{MacMeachan}. There
$X=(\T,\ell)$ was an irreducible principally polarized abelian surface
(ppas for short) and so $\rho(X)=1$ and $g=2=\tilde g$. Note that we must have $\gamma=0$
(and so $d=0$ in he formulae of Note \ref{n:summary}). We set
$v=(1,2\ell,4-k)$ and $\beta=\omega$. Then $F=k$ and the bounds
become
\[\frac{3-k}{2}\leq C<2-\sqrt{k},\]
\[R\leq \frac{1}{2}(k-1),\]
and the walls are constrained to satisfy $2-k\leq s<1$. It follows
that $L^2\I_p$ for a point $p$ is $\sigma$-stable for all
$\sigma\in\Plane{\omega,\omega,0}$ for which the sheaf is in
$\A$. Note that, in that case, the maximal radius is achieved (by
$w=(1,\ell,1)$). This agrees with the finding of \cite{MacMeachan}.

In that paper (and elsewhere) it is shown that the twisted ideal sheaf
$L\otimes \I_V$ is $Z$-stable for
any zero-scheme $V$ but the situation is very different if we assume
$\T$ is reducible. Suppose $\T=E_1\times E_2$ is a product of elliptic
curves (for simplicity, we assume they are non-isomorphic but this is
not really necessary). Then the canonical polarization splits as
$\ell=\omega=\ell_1+\ell_2$. In this case, we have $\tilde g=1$ and so
$g'=2$. This time we must pick $m=3$ and $n=1$. We still have $F=k$
but now $s<3/2$ for our walls and our maximal radius (of a
pseudo-wall) is now $k-1/4$ with centre $7/4-k$. But this does not
correspond to a real wall. The maximal wall is given by
$w=(1,2\ell_1+\ell_2,2)$ with centre $2-k$ and radius $\sqrt{k(k-1)}$. Then
  $s<2-k+\sqrt{k(k-1)}<3/2$, for all $k$. In fact, the limit of $C+R$
  for these walls as $k\to\infty$ is $3/2$ but the limit is never achieved. 

To give a more complete picture, consider the same product torus but
this time let $v=(1,\ell_1+\ell_2,1-k)$. This corresponds to the
sheaves $L\otimes\I_V\otimes \P_\xhat$, where $\P_\xhat\in\Pic^0\T$
and $|V|=k$. These sheaves have two families of potentially
destabilizing subsheaves: $L_1\otimes\I_{V'}$ and $L_2\otimes\I_{V'}$
up to twists by flat line bundles, 
where $|V'|<|V|$. These correspond geometrically to $V'\subset V$ and
$V\setminus V'$ being contained in a translate of either $E_2$ or
$E_1$, respectively. These two families of walls intersect along a plane
of constant $u$ (as predicted by Theorem \ref{t:nestedwalls}). The
situation is illustrated in Figure 2 for the case $k=4$, where the vertical axis is $t$
and the two grey planes are $s=1$ and $u=0$. On the plane $u=0$ the
walls have centres at $s=m-3$ and radii $\sqrt{m^2-8m+12}$. These
radii are positive for $m=0$ and $m=1$ only. Note that in the $(s,u)$
plane the walls are parabolas and so are unbounded.

\begin{figure}
\includegraphics[height=4in]{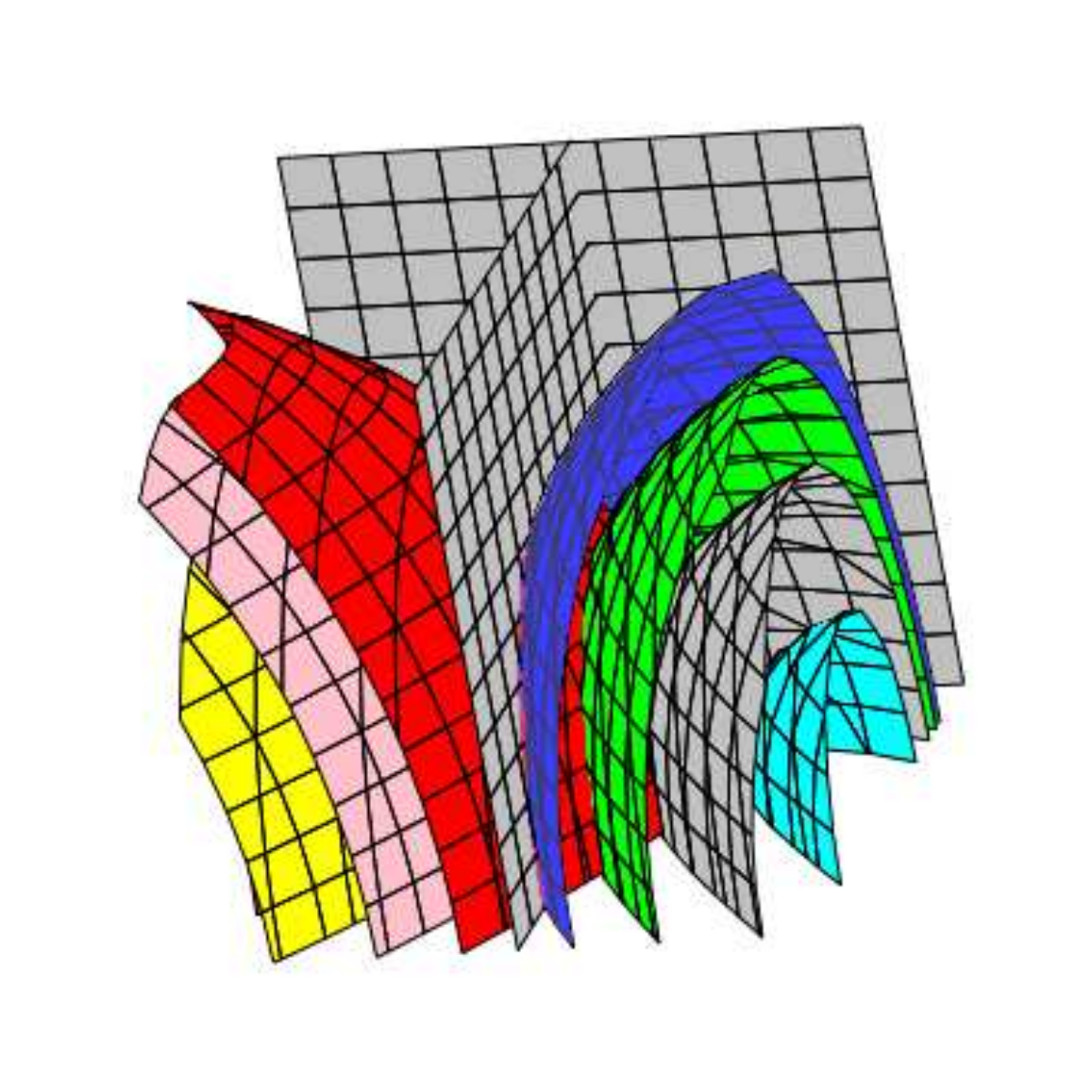}
\caption{Walls in $\Plane{\ell_1,\ell,u}$}
\end{figure}

\subsection{Global finiteness.}
We can ask if the set of walls is always finite but two examples in
\cite{MeachanThesis} show that this is not the case. For an
irreducible principally polarized abelian surface $(\T,\omega)$ the Chern characters
$(1,2\omega,2)$ and $(1,2\omega,1)$ result in infinite families of
walls. But it does confirm our finding that, away from $C_0$, each ray
contains only finitely many walls. However, we can prove:
\begin{prop} In the case of a principally polarized abelian surface $(\T,\omega)$,
if $q$ is rational then the number of mini-walls in the ray $\beta=q\omega$ is finite.
\end{prop}
\begin{proof}
This follows because for such $q$, there is a Fourier-Mukai transform
$\Phi_q$ which reverses the ray (see \cite[Prop
3.2]{MacMeachan}). Then the proposition follows from Theorem \ref{t:bound}.
\end{proof}

As was shown by Bridgeland in \cite{BrK3}, the space of all (good) Bridgeland
stability conditions $\stab(\T)$ is $3$-dimensional and each stability
condition can be written as $\Phi(\A,Z)$ for some Fourier-Mukai
transform $\Phi$. Since $\Phi$ preserves the chamber and wall
structure we have
\begin{thm} Let $(\T,\omega)$ be an irreducible  principally polarized
  abelian surface. The walls in $\stab(\T)$ are nested real hypersurfaces.
\end{thm}
\begin{remark}
This will hold more generally for abelian surfaces and K3 surfaces
when stated properly but we omit the details.
\end{remark}
\begin{example}
In particular, by Theorem \ref{t:finite}, if $y_1/x-\sqrt{F}$ is rational (ie, when $F$ is a
square rational) then the number of walls in $\stab(X)$ is (globally)
finite. For example, on a ppas $(\T,\omega)$ with $v=(5,3\omega,1)$ we have
$F=4/25$ and then the critical ray is $s=1/5$. 
\end{example}

\subsection{The case of O'Grady's Holomorphic Symplectic Space}
Let us now look at an example where $p\neq1$. In \cite{OGrady}, Kieran
O'Grady showed that the moduli space of Gieseker stable sheaves with
Chern character $(2,0,-2)$ (modulo natural torus actions) gives a
simply connected symplectic holomorphic manifold which was not one of
the (then) known types. 
\begin{prop}\label{p:ogrady}
Let $(X,\omega)$ be an irreducible principally polarized abelian
surface. Then the Chern character $v=(2,2\omega,0)$ has no walls in
$\stab(X)$.
\end{prop}
\begin{proof}
Since $\stab(X)$ is generated by Fourier-Mukai transforms it suffices
to prove this for basic Bridgeland stability conditions. Furthermore, it suffices
to show that any $\mu_\omega$-semistable sheaf $E$ which is also
$\mu_Z$-stable for some $Z$ is $\mu_Z$-stable for all $Z$.
Note that $F=1$ for this choice of $v$ and $C_0=0$. So Remark
\ref{r:cbound} gives
$0<c<2$
and then $c_1(K)=\omega$ for any destabilizing object $K$ in any
$\A_s$. We also must have $r(K)>1$ because otherwise $xc_1-y_1r=0$. As
usual, we can assume that $K$ is $\mu_\omega$-semistable. On
the other hand, the  centre  of the wall 
corresponding to $(r,\omega,k)$ is
$\dfrac{-\chi}{2r-2}<C_0=0$.
But then $\chi>0$ and that contradicts the Bogomolov inequality for $K$.
\end{proof}

\subsection{A Conjecture}
The examples above and Theorem \ref{t:nestedwalls} and its corollary
suggest the following:
\begin{conjecture}
For any smooth complex projective surface $X$ and any $v\in N(X)$ the
complement of the walls in the space of good stability conditions
union the boundary of its closure is path-connected.
\end{conjecture}
In other words, any $\sigma$-stable moduli spaces survives to the
boundary of $\stab(X)$. The conjecture holds for abelian surfaces and
holds for the subspace of basic stability conditions for any surface.

%%% Local Variables: 
%%% mode: latex
%%% TeX-master: "p60_jag"
%%% End: 

\bibliographystyle{alpha}
\bibliography{p60}

\end{document}